\documentclass[a4paper, reqno, 14pt]{amsart}

\usepackage[usenames,dvipsnames]{color}
\usepackage{amsthm,amsfonts,amssymb,amsmath,amsxtra}
\usepackage[all]{xy}
\SelectTips{cm}{}
\usepackage{xr-hyper}
\usepackage[colorlinks=true, allcolors=blue]{hyperref}
\usepackage{verbatim}
\usepackage{enumerate}
\usepackage{cleveref}

\usepackage[margin=1.25in]{geometry}
\usepackage{mathrsfs}
\usepackage{float}
\usepackage{tikz}
\usetikzlibrary{positioning}

\definecolor{myred}{rgb}{0.8,0,0}
\definecolor{myblue}{rgb}{0,0,0.8}
\definecolor{mygreen}{rgb}{0,0.6,0}

\newcommand{\drawHexFractal}[4]{%
    \begin{scope}[shift={({#1},{#2})}]
        \fill[gray!60] (30:{#3}) \foreach \a in {90,150,210,270,330,30} { -- (\a:{#3}) };
        \foreach \i/\col in {0/myred,1/mygreen,2/myblue,3/myred,4/mygreen,5/myblue} {
            \pgfmathsetmacro{\a}{30+60*\i}
            \pgfmathsetmacro{\b}{30+60*(\i+1)}
            \draw[line width=1pt,\col] (\a:{#3}) -- (\b:{#3});
        }
        \foreach \i in {0,...,5} {
            \pgfmathsetmacro{\a}{30+60*\i}
            \fill[black] (\a:{#3}) circle (1.3pt);
        }
        \ifnum#4>0
            \foreach \i in {0,...,5} {
                \pgfmathsetmacro{\angle}{30+60*\i}
                \pgfmathsetmacro{\cx}{2*#3*cos(\angle)}
                \pgfmathsetmacro{\cy}{2*#3*sin(\angle)}
                \draw[mygreen,thick] (\angle:{#3}) -- (\cx,\cy);
                \begin{scope}[shift={(\cx,\cy)}]
                    \pgfmathsetmacro{\r}{0.45*#3}
                    \pgfmathsetmacro{\parentCorner}{mod(\angle+180,360)}
                    \foreach \j in {0,...,5} {
                        \pgfmathsetmacro{\theta}{30+60*\j}
                        \ifdim\theta pt=\parentCorner pt\relax\else
                            \draw[thick,black] (\theta:\r) -- ++(\theta:{0.5*\r});
                        \fi
                    }
                \end{scope}
                \drawHexFractal{\cx}{\cy}{0.45*#3}{\numexpr#4-1\relax}
            }
        \fi
    \end{scope}
}

\RequirePackage{xspace}
\RequirePackage{etoolbox}
\RequirePackage{varwidth}
\RequirePackage{enumitem}
\RequirePackage{tensor}
\RequirePackage{mathtools}
\RequirePackage{longtable}
\RequirePackage{multirow}

\def\d{\delta}
\def\i{{^{-1}}}
\def\subset{\subseteq}

\def\NN{\mathbb{N}}
\def\co{\mathcal{O}}
\def\ch{\mathcal{H}}
\renewcommand{\int}{\operatorname{int}}

\DeclareMathOperator{\supp}{supp}
\DeclareMathOperator{\Ad}{Ad}

\DeclareMathOperator{\cl}{cl}
\DeclareMathOperator{\Ch}{Ch}
\DeclareMathOperator{\Cay}{Cay}
\DeclareMathOperator{\Stab}{Stab}
\DeclareMathOperator{\Aut}{Aut}
\DeclareMathOperator{\id}{id}
\DeclareMathOperator{\fin}{fin}

\newtheorem{thm}{Theorem}[section] 
\newtheorem{prop}[thm]{Proposition}
\newtheorem{lem}[thm]{Lemma}
\newtheorem*{conj}{Conjecture}
\newtheorem{cor}[thm]{Corollary}
\newtheorem{namedthm}{Theorem}
\newenvironment{thmA}[1]{%
  \renewcommand\thenamedthm{#1}%
  \begin{namedthm}
}{%
  \end{namedthm}
}

\theoremstyle{definition}

\newtheorem{defn}[thm]{Definition}
\newtheorem{example}[thm]{Example}
\newtheorem*{ass}{Assumption}

\theoremstyle{remark}

\newtheorem*{rmk}{Remark}
\newtheorem*{claim*}{Claim}

\numberwithin{equation}{section}
\numberwithin{thm}{section}

\usepackage{enumitem}
\setlist[enumerate,1]{label=(\roman*), ref=\roman*}

\newcommand\restr[2]{\left.\kern-\nulldelimiterspace #1 \vphantom{\big|} \right|_{#2}}

\title{Centralizers in Hecke Algebras of Any Coxeter Group}
\author{Haiyu Chen}

\begin{document}
\maketitle
\begin{abstract}
We study the centralizer of a parabolic subalgebra in the Hecke algebra associated with an arbitrary (possibly infinite) Coxeter group. While the center and cocenter have been extensively studied in the finite and affine cases, much less is known in the indefinite setting. We describe a basis for the centralizer, generalizing known results about the center. Our approach combines algebraic techniques with geometric tools from the Davis complex, a CAT($0$)-space associated to the Coxeter group. As part of the construction, we classify finite partial conjugacy classes in infinite Coxeter groups and define a variant of the class polynomial adapted to the centralizer.
\end{abstract}

\section{Introduction}
\subsection{Background and Motivation}
Hecke algebras naturally arise in representation theory and play a central role in understanding the representations of finite groups of Lie type and $p$-adic groups. This paper studies a general version of Hecke algebras associated with arbitrary Coxeter groups.

Let $(W,S)$ be a Coxeter system with finite generating set $S$ and length function $\ell$. Let $R$ be a commutative ring. Suppose we are given a family of parameters $(a_s, b_s)_{s \in S}$ in $R$ such that if $s$ and $t$ are conjugate in $W$, then $(a_s, b_s) = (a_t, b_t)$. The \emph{generic Hecke algebra} $\mathcal{H} = \ch(W, S, (a_s, b_s)_{s \in S})$ is defined as the free $R$-module with basis $\{T_w\}_{w \in W}$ and multiplication rules
\[
\begin{aligned}
T_v T_w &= T_{vw} && \text{if } \ell(vw) = \ell(v) + \ell(w), \\
T_s^2 &= a_s T_s + b_s && \text{for all } s \in S.
\end{aligned}
\]
A notable special case is the Iwahori-Hecke algebra, obtained by setting $a_s = q - 1$ and $b_s = q$ for all $s \in S$ and a prime power $q$. Throughout, we assume all $b_s$ are invertible in $R$.

A guiding philosophy in representation theory is that ``characters tell everything.'' In analogy with group algebras, the character table of a Hecke algebra governs its irreducible representations. Structural subgroups such as the center and cocenter play pivotal roles: the center $Z(\mathcal{H})$ controls the number of irreducible representations, while the cocenter $\mathcal{H}/[\mathcal{H},\mathcal{H}]$ governs the space of trace functions. 

To generalize the notion of the center, we consider the centralizer of a parabolic subalgebra. Given $J \subseteq S$, let $\ch_J:=\mathcal{H}(W_J,J,(a_s,b_s)_{s\in J})$ be the subalgebra of $\ch$ generated by $\{T_s\}_{s \in J}$. The \emph{centralizer} of $\ch_J$ in $\ch$ is defined as
\[
\mathcal{Z}_{\ch}(\ch_J) := \{ h \in \ch \mid h x = x h \text{ for all } x \in \ch_J \}.
\]
When $J = S$, this recovers the center $Z(\mathcal{H})$. 

A comparison between the center and cocenter across different sizes of Coxeter groups reveals deep structural trends. In the finite case, they are both of rank $\sharp\cl(W)$ by \cite{geck2000characters}. For the affine case, the center has a basis indexed by dominant weights $P^+$ \cite{lusztig1989affine}, which are in bijection with the conjugacy classes of translations. It is further in bijection with the set of finite conjugacy classes in $W$. The cocenter \cite{he2014minimal} is indexed by all conjugacy classes $\cl(W)$. In the infinite non-affine case, the center is trivial \cite{marquis2024centre}, and the cocenter is conjectured as below. 
\begin{table}[H]
\centering
\begin{tabular}{|l|c|c|}
\hline
\textbf{Dimension} & \textbf{Center} & \textbf{Cocenter} \\
\hline
Finite Case & $\sharp \cl(W)$ & $\sharp \cl(W)$ \\
\hline
Affine Case & $\sharp P^+ = \sharp \cl^{\mathrm{fin}}(W)$\footnotemark & $\sharp \cl(W)$ \\
\hline
Infinite Non-Affine Case & Trivial & $\sharp \cl(W)$\textbf{\textcolor{red}{??}} \\
\hline
\end{tabular}
\caption{Center and cocenter comparison across different Coxeter group types}
\end{table}

This comparison indicates that when the group $W$ is larger, the center shrinks and the cocenter expands.  This work and \cite{partialcocenterADLV} on the following generalization reinforce this observation:

\begin{table}[h]
\centering
\begin{tabular}{|l|c|c|}
\hline
\textbf{Dimension} & \textbf{$\mathcal{Z}_{\mathcal{H}}(\mathcal{H}_J)$} & \textbf{$\mathcal{H}/[\mathcal{H}, \mathcal{H}_J]$} \\
\hline
General Coxeter Groups & $\sharp \cl_J^{\mathrm{fin}}(W)$ (\cref{thm:mainthmcentralizer}) & $\sharp \cl_J(W)$ (\cite{partialcocenterADLV}) \\
\hline
\end{tabular}
\caption{Comparison of partial centralizers and cocenters}
\end{table}

\subsection{Problem Setup and Main Results}

Let $\ch$ and $\ch_J$ be as above. The structure of $\mathcal{Z}_{\mathcal{H}}(\mathcal{H}_J)$ is investigated in three main steps, each centered around a key result:

\subsubsection*{Step 1: Classification of Finite $W_J$-Conjugacy Classes}

We first determine exactly which elements $w \in W$ have finite $W_J$-conjugacy classes. This is formalized in the following theorem, which may be of independent interest.

\begin{thmA}{A}[Classification of Finite $W_J$-Conjugacy Classes]
Suppose $(W,S)$ is an infinite Coxeter group and $J$ is irreducible. Then $\co_J(w)$ is finite only in the following cases:
\begin{enumerate}
    \item $J$ is spherical, i.e., $W_J$ is finite;
    \item $w \in W_{J^\perp}$;
    \item $J$ is affine, and $w = w_1 w_2$, $w_1 \in W_{J^\perp}$, $w_2 \in W_J$, and $w_2$ is a translation in $W_J$.
\end{enumerate}
Here, $J^\perp = \{ s \in S \setminus J \mid sj = js \text{ for all } j \in J \}$.
\end{thmA}

\subsubsection*{Step 2: Maximal Length Elements and Geometric Finiteness}

Two important concepts for analyzing Coxeter group conjugacy classes are cyclic shift and strong conjugation.

Roughly speaking, a cyclic shift on $w \in W$ means moving a generator from the start of a reduced expression of $w$ to the end or vice versa. We denotes $w \xrightarrow{s} w'$ if $w'=s w s$ and $\ell(w') \le \ell(w)$, $s \in S$, and denote $w \rightarrow w'$ for a sequence of such arrows.

Strong conjugation means moving the first few generators of a reduced expression of $w$ to the end while preserving the length, or vice versa. We denote strong conjugation by $w \sim w'$ when $\ell(w)=\ell(w')$ and there exists $x \in W$ such that $w'=x w x^{-1}$ and $\ell(x w)=\ell(x)+\ell(w)$ or $\ell(w x^{-1})=\ell(x)+\ell(w)$. 

A fundamental theorem for Coxeter groups states that every element can be cyclically shifted to minimal length in its conjugacy class, and any two minimal length elements are strongly conjugate \cite{marquis2018cyclically}. In this work, we develop the maximal version of this result:

\begin{thmA}{B}[Maximal Length Representatives]
Let $(W,S)$ be a Coxeter system. Let $\co_{J}$ be a $W_J$-conjugacy class of $W$. Then the following assertions hold:
    \begin{enumerate}
         \item If $\co_{J}$ is finite, then for each $u \in \co_{J}$, there exists a finite chain $u_k\xrightarrow{s_k}\cdots\xrightarrow{s_1}u$, where $u_k \in \co^{\max}_J$ and all $s_j \in J$.
         \item If $\co_{J}$ is infinite, then for each $u \in \co_{J}$, there exists an infinite chain of distinct elements $\cdots\xrightarrow{s_2}u_1\xrightarrow{s_1}u$, where all $s_j \in J$.
         \item For any $u, u' \in \co^{\max}_{J}$, we have $u \backsim_{J} u'$. (definition in \ref{defmaxequ})
    \end{enumerate}
\end{thmA}

One important step in dealing with infinite conjugacy classes is using metric geometry. We study the set
\[
U^+_J(w) := \{ w' \in W \mid w' \to_J w \}
\]
and show that it is finite if and only if $\co_J(w)$ is finite. This is achieved by applying tools from geometric group theory via the Davis complex.

The Coxeter complex $\Sigma$, constructed from cosets $\{wW_T\}$ and simplicial inclusions, encodes much of the algebraic structure of $W$. However, it is not locally finite when $W$ is infinite, preventing properly discontinuous group actions. Davis \cite{davis1994buildings} introduced a metric realization of $\Sigma$ into a complete CAT(0) space, called the \emph{Davis complex} $X$. This space is piecewise Euclidean, contractible, and admits a proper, cocompact action of $W$ by isometries. Convexity properties of displacement functions on $X$—a consequence of the CAT(0) geometry—are essential to our proof that $U^+_J(w)$ and $\co_J(w)$ have matching finiteness behavior.

\subsubsection*{Step 3: Construction of Centralizer Basis}

If $\co_J(w)$ is infinite, then the entire conjugacy class contributes nothing to the centralizer. Otherwise, we inductively define \( f^{\max}_{w,\mathcal{O}_J}\in R \), a variant of the class polynomials of Geck-Pfeiffer \cite{geck1993irreducible}, and construct a basis element for each finite $W_J$-conjugacy class.

\begin{thmA}{C}[Basis of the Centralizer]
The centralizer $\mathcal{Z}_{\mathcal{H}}(\mathcal{H}_J)$ is a free $R$-module. An $R$-basis is given by
\[
\left\{ z_{\mathcal{O}_J} = \sum_{w \in W} b_w^{-1} f^{\max}_{w, \mathcal{O}_J} T_{w^{-1}} \right\}_{\mathcal{O}_J},
\]
where the index $\mathcal{O}_J$ ranges over all \textbf{finite} $W_J$-conjugacy classes in $W$. Each sum is finite.
\end{thmA}

\subsubsection*{Conjecture on the Cocenter}

Motivated by our work, we propose the following conjecture, consistent with known results in the finite and affine settings \cite{geck1993irreducible, he2014minimal}, and a partial result is known in \cref{prop:cocenterbasis}:

\begin{conj}
The cocenter $\mathcal{H}/[\mathcal{H}, \mathcal{H}]$ has a basis indexed by the set of conjugacy classes $\cl(W)$.
\end{conj}

\subsection{Acknowledgement}
I am deeply grateful to Xuhua He for suggesting this problem and for his invaluable guidance throughout the project. I also thank Michael McBreen for his consistent support during my Ph.D. studies. I am sincerely thankful to Felix Schremmer for the insightful discussions and helpful suggestions.

\section{Preliminaries}

    Let $(W,S)$ be a Coxeter system with $S$ a finite set. The Coxeter group $W$ is equipped with the length function $\ell: W \to \NN$ and the Bruhat order $\le$. 
    
    The subgroups $W_J:=\langle J\rangle\subseteq W$ ($J\subseteq S$) are called \emph{standard parabolic subgroups} of $W$, and their conjugates are \emph{parabolic subgroups} of $W$. For $w\in W$, each $W_J$-coset $W_Jw$ contains a unique minimal length element $\min W_J w$. Let $^J W=\{\min W_J w\ |\ w\in W\}$. Moreover, $\ell( yx)=\ell(y)+\ell(x)$ for any $x\in {}^JW$, $y \in W_J$. The above notions can be similarly defined for $W^J$. For $K,K'\subset S$, we denote ${}^{K'}W^K:={}^{K'}W\cap W^K$.
    
        We say $J\subset S$ \emph{spherical} if $W_J$ is finite; otherwise we say $J$ is non-spherical. The reason is that a building is spherical if and only if its Weyl group is a finite Coxeter group. See, for example, \cite{abramenko2008buildings}. We say $J$ is \emph{affine} if $W_J$ is affine. We say $J$ (resp. $W_J$) is \emph{irreducible} if it does not decompose nontrivially as a disjoint union $J=J_1\sqcup J_2$ with $s_1s_2=s_2s_1$ for all $s_1\in J_1$ and $s_2\in J_2$. Denote $J^{\perp}:=\{s\in S\setminus J \ | \ st=ts \ \forall t\in J\}$. For $w\in W$, let $w\cdot J:=w J w \i$.

\subsection{Partial Cyclic Shift Classes}
    In 1993, Geck-Pfeiffer \cite{geck1993irreducible} pointed out that two special kinds of conjugation action are important in the study of Hecke algebras, namely cyclic shift and strong conjugation.
    Let $J\subset S$ be a subset.
    \begin{defn}
    
    \begin{enumerate}
        \item For $w, w' \in W$, we write $w \xrightarrow{s} w'$ if $w'=s w s$ and $\ell(w') \le \ell(w)$, $s \in S$; We write $w \xrightarrow{J} w'$ if there exists a sequence $w=w_0, w_1, \cdots, w_n=w'$ of elements in $W$ such that for any $k$, $w_{k-1} \xrightarrow{s} w_k$ for some $s \in J$.  We write $w \approx_{J} w'$ if $w \xrightarrow{J} w'$ and $w' \xrightarrow{J} w$, equivalently, $w \xrightarrow{J} w'$ and $\ell(w)=\ell(w')$.

        \item We call $w, w' \in W$ \emph{elementarily strongly $J$-conjugate} if $\ell(w)=\ell(w')$ and there exists $x \in W_J$ such that $w'=x w x^{-1}$ and $\ell(x w)=\ell(x)+\ell(w)$ or $\ell(w x^{-1})=\ell(x)+\ell(w)$. We call $w, w'$ {\it strongly $J$-conjugate} if there is a sequence $w=w_0, w_1, \cdots, w_n=w'$ such that for each $i$, $w_{i-1}$ is elementarily strongly $J$-conjugate to $w_i$. We write $w \sim_{J} w'$ if $w$ and $w'$ are strongly $J$-conjugate.

    \end{enumerate}  
    \end{defn}
    \begin{rmk}
    \begin{enumerate}
        \item In all notation above, we omit $J$ if $J=S$;
        \item Note that $\approx_J$ implies $\sim_J$.
    \end{enumerate}
        Both $\approx_{J}$ and $\sim_{J}$ are equivalence relations. The equivalence classes under $\approx_{J}$ are called the \emph{$J$-cyclic shift classes}. 
        
    \end{rmk}



\begin{defn} We set
    $U^+_J(w):=\{w'\in W \ | \ w'\xrightarrow{J} w\}.$
\end{defn}
Note that $U^+_J(w)\subset W_J\cdot w$.

An important result for the calculation is the following. Here we only review the classical result for the case of finite Coxeter groups and $J=S$. 

\begin{thm}[\cite{geck1993irreducible}]\label{thm:minimal_thm_finite}
    Let $(W,S)$ be a finite Coxeter group. Let $\co^{\min}$ be the set of minimal length elements in $\co$. Then
    \begin{enumerate}
        \item For any $w\in W$, $w\rightarrow w'$ for some $w'\in\co^{\min}$.
        \item For any two $w,w'\in\co^{\min}$, $w\sim w'$.
    \end{enumerate}
\end{thm}
It is later generalized to extended affine Weyl groups \cite{he2014minimal} and we will review this result for general Coxeter groups later (see \cref{thm:mingeneral}).

\subsection{Davis complex}
Basics on Davis complexes can be found in \cite{davis1994buildings,davis2008geometry} and also in \cite{marquis2020structure}.
\subsubsection{Coxeter complex}
A Coxeter complex $\Sigma=\Sigma(W,S)$ is a simplicial complex with simplices being the cosets $\{wW_I\ |\ w\in W,I\subsetneq S \}$, with face relation being opposite to the usual set inclusion $\subset$. The maximal simplices $wW_\emptyset=\{w\}$ are called \emph{chambers}. Denote by $\Ch(\Sigma)$ the set of chambers. The \emph{fundamental chamber} $C_0$ is the chamber corresponding $\{e\}$. Chambers are in bijection with elements in $W$. We sometimes use elements in $W$ to represent a chamber.

Any two \emph{adjacent} chambers are of the form $\{w\}$ and $\{ws\}$, where $w\in W, s\in S$, contains a \emph{panel} $m_s$ of type $s$, which is the 1-codimensional intersection of the chambers. In this case, we say they are \emph{$s$-adjacent}. A gallery $\Gamma=(D_0,D_1,\dots,D_k)$ of length $k$ is a sequence of chambers such that $D_{i-1}$ and $D_i$ are $s_i$-adjacent. We say $(s_1,s_2,\dots,s_k)$ is the \emph{type} of $\Gamma$. If $s_i\in J\subset S$ for all $i$, then we call $\Gamma$ to be a \emph{$J$-gallery}.

For any two chambers $C,D\in \Ch(\Sigma)$, the galleries from $C$ to $D$ are called \emph{minimal} if their lengths are minimal among all galleries from $C$ to $D$. We define the chamber distance $d_{\Ch}(C,D)$ to be the length of the minimal galleries between $C$ and $D$. For $C=vC_0$, $D=wC_0$, we have $d_{\Ch}(C,D)=\ell(v\i w)$. Now fix $w\in W$. Chambers in Coxeter complex encodes information on conjugates of $w$ through the following function $\pi_w:\Ch(\Sigma)\to W$, $v C_0\mapsto v\i wv$. For $D\in \Ch(\Sigma)$, the chamber distance $d_{\Ch}(D,wD)=\ell(\pi_w(D))$.

The group $W$ acts on $\Sigma$ by left multiplications, which are simplicial isometries of $\Sigma$. Elements of $s\in S$ act by \emph{simple reflections}, and their conjugates are \emph{reflections}. Each reflection $r=r_m$ has a fixed-point set $m$ called a \emph{wall}.  

We may visualize $\Sigma$ as a simplicial complex with the fundamental domain being a simplex of dimension $|S|-1$. It is identified with the fundamental chamber described above. Each coset $wW_I$ is a simplex of codimension $|I|$ (or dimension $|S|-|I|-1$).  For each $w\in W$ we attach a copy of the fundamental chamber and glue them by face relations. We may also draw a Cayley graph $\Cay(W,S)$ on $\Sigma$ as follows. A vertex $\{v\}$ is the barycenter $x_D$ of the chamber $D=vC_0$. An $s$-edge is drawn between $x_C$ and $x_D$ for $C,D\in \Ch(\Sigma)$ that are $s$-adjacent. 

Every wall divides $\Sigma$ into two \emph{half-spaces}: the graph $\Cay(W,S)\setminus m$ has two connected components whose vertex sets are the barycenters of chambers in these half-spaces. Walls can be visualized as 1-codimensional hyperplanes in the simplicial complex. Every panel is contained in a unique wall $m$: The $s$-adjacent chambers $wC_0$ and $wsC_0$ contain a panel $m_s$, which is contained in a wall $m$ corresponding to reflection $r_m=wsw\i$. We say the wall $m$ \emph{separates} chambers $C,D\in \Ch(\Sigma)$ if they are in different half-spaces associated with $m$. We also say the $|S|$ walls $\{m\ |\ r_m=wsw\i,\ s\in S\}$ are walls of $wC_0$. The chamber distance $d_{\Ch}(C,D)$ equals the number of walls separating $C$ and $D$. 

We may read off cosets $wW_J$ from $\Sigma$ by the notion of a $J$-residue. For a chamber $D=wC_0$, the set of chambers $R_J(D)$ connected to $D$ by a $J$-gallery is called a \emph{$J$-residue} of $D$. We see that $R_J(wC_0)$ corresponds to elements in $wW_J$. Its stabilizer in $W$ is $\Stab_W(R_J(wC_0))=wW_Jw\i$. A wall $m$ separates two chambers of a residue $R$ if and only if $r_m\in \Stab_W(R)$, in which case we call it a \emph{wall of $R$}.

\subsubsection{Davis complex}
The Davis complex $X=X(W,S)$ is a metric realization of Coxeter complex $\Sigma(W,S)$. It is a complete CAT($0$) cellular complex with piecewise Euclidean metric $d$, with a natural action of $W$ by cellular isometries. While the fundamental domain of a Coxeter complex is a simplicial complex of dimension $|S|-1$, that of a Davis complex is usually not. We now describe a cellular construction of the Davis complex as follows. 

Let $I\subset S$ be a \emph{spherical} subset, meaning the parabolic group $W_I$ is finite. The $|I|$-dimensional cell $P_I$ is called a Coxeter cell, defined as follows. $W_I$ acts on an $I$-dimensional Euclidean space $V_I$ with simple reflections $s\in I$ acting by linear reflections. The wall $m_s$ is a hyperplane fixed by $s$. The collection of walls $(m_s)_{s\in I}$ delimit a simplicial cone $C_I$ in $V_I$. Let $x_I$ be the point in the interior of $C_I$ with distance $1$ to each wall $m_s$. Then $P_I$ is the convex hull of points $W_Ix_I$, with induced Euclidean metric from $V_I$. Note that if $I\subset J$ are both spherical, then $P_I$ is a face of $P_J$.

The $1$-skeleton $X^{(1)}$ is the Cayley graph of $(W,S)$. For each $I\subset S$ spherical, we glue $I$-dimensional cells $P_I$ to the Cayley subgraph of $(W_I,I)$ inductively from low to top dimensions according to face relations above. The metric $d$ is obtained by gluing each Euclidean metric on $P_I$, for $I\subset S$ spherical. Since $W$ acts on the Cayley graph preserving spherical residues, it induces an action on $X$ by cellular isometries.

The chamber system of Coxeter complexes can also be realized in Davis complexes. We realize chambers of $\Sigma$ inside $X$ as a closed convex subset. For $C=vC_0$, and let $x_C:=\{v\}$ be a vertex in $X^{(1)}$. For each spherical residue $wR_I$, $I\subset S\text{ spherical}$ containing $C$ corresponds to a polytope $P_I$ in $X$ with $x_C=x_I$. Define $C:=\bigcup_{I\subset S\text{ spherical}} P_I\cap C_I$. We refer to $x_C$ as the \emph{barycenter} of $C$. We identify chambers and walls of $\Sigma$ with their realization of $X$. We also set $\Ch(X)=\Ch(\Sigma)$ to be the collection of chambers in $X$.

\begin{example}
The Davis complex of the Coxeter group
\begin{tikzpicture}[scale=0.6,baseline={(current bounding box.center)}]
  \node[circle, draw, inner sep=1pt, label=above:$s$] (s) at (0,1) {};
  \node[circle, draw, inner sep=1pt, label=left:$t$] (t) at (-1,0) {};
  \node[circle, draw, inner sep=1pt, label=right:$u$] (u) at (1,0) {};
  \draw (s) -- (t) node[midway, above left=0.5pt and 1pt] {\tiny 3};
  \draw (s) -- (u) node[midway, above right=0.5pt and 1pt] {\tiny $\infty$};
  \draw (t) -- (u) node[midway, below=2pt] {\tiny $\infty$};
\end{tikzpicture}
has the fundamental chamber
\begin{tikzpicture}[baseline={(current bounding box.center)}]
  \coordinate (top) at (0,0.4);
  \coordinate (right) at (0.3,0);
  \coordinate (bottom) at (0,-0.3);
  \coordinate (left) at (-0.3,0);
  \coordinate (c0) at (0,-0.3);
  \filldraw[fill=black!50, thick] (top) -- (right) -- (bottom) -- (left) -- cycle;
  \fill (c0) circle (1pt);
  \node[right] at (c0) {\tiny $x_{C_0}$};
  \draw[thick] (c0) -- ++(0,-0.2);
\end{tikzpicture},
which is shaded in the picture below.
It is given by $C=\bigcup_{I\subset S\text{ spherical}} P_I\cap C_I$. The hexagon in the picture is $P_{\{s,t\}}$. The three lines are walls of three generators. The Davis complex is shown below.

\medskip

\noindent
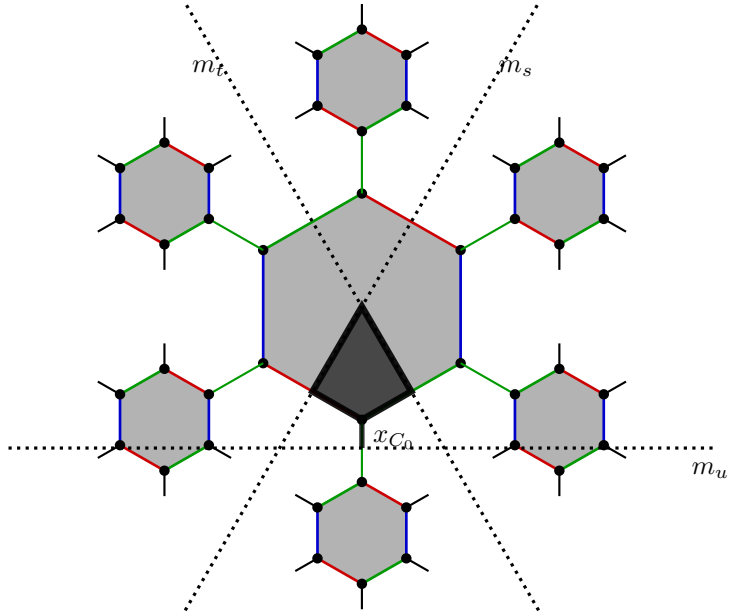
\begin{figure}[h]
    \centering
    \begin{tikzpicture}[scale=1.5]
        \drawHexFractal{0}{0}{1}{1}
        \draw[dotted,very thick] ({-3.1*cos(60)}, {-3.1*sin(60)}) -- ({3.1*cos(60)}, {3.1*sin(60)});
        \draw[dotted,very thick] ({-3.1*cos(120)}, {-3.1*sin(120)}) -- ({3.1*cos(120)}, {3.1*sin(120)});
        \draw[dotted,very thick] (-3.1,-1.25) -- (3.1,-1.25);
        \node at ({2.3*cos(60)}, {2.3*sin(60)}) [above right=-2pt] {$m_s$};
        \node at ({2.3*cos(120)}, {2.3*sin(120)}) [above left=-2pt] {$m_t$};
        \node at (0, -1.25) [below right=2pt and 120pt] {$m_u$};
        \node at (0, -1) [below right=1pt] {$x_{C_0}$};
        \fill[black, opacity=0.6]
            (0, 0)
            -- (0.433, -0.75)
            -- (0, -1)
            -- (-0.433, -0.75)
            -- cycle;
        \draw[line width=2.5pt, black, opacity=0.9]
            (0, 0)
            -- (0.433, -0.75)
            -- (0, -1)
            -- (-0.433, -0.75)
            -- cycle;
        \draw[black, line width=2pt, opacity=0.8]
            (0, -1) -- (0, -1.25);

    \end{tikzpicture}
    \caption{Davis complex of the Coxeter group ($3,\infty,\infty$).}
\end{figure}

\end{example}

\subsection{Class Polynomials}
Class polynomials arise from studying the cocenter $\ch/[\ch,\ch]$. The theory of conjugacy classes of finite Coxeter groups can be found in \cite[\S 8]{geck2000characters}. We will summarize some important results for a finite Coxeter group $(W,S)$ here.

\begin{lem}\label{lem:sameimg}
    If $w\sim w'$, then $T_w\equiv T_{w'} \mod [\ch,\ch]$.
\end{lem}
\begin{proof}
    Write $w'=xwx\i$ with $\ell(w)=\ell(w')$ and $x\in W_J$. We have either $\ell(xw)=\ell(w)+\ell(x)$ or $\ell(wx\i)=\ell(w)+\ell(x)$. As we assume $\{b_s\}_{s\in S}$ are invertible, $T_x$ is invertible. In the first case, $T_{w'}=T_x T_w T_{x}\i$. Hence $T_w-T_{w'}= T_w-T_x T_w T_{x}\i=(T_wT_x-T_xT_w)T_x\i\in [\ch,\ch]$. The second case is similar.
\end{proof}

        Let $\cl(W)$ be the set of all conjugacy classes in $W$. For each finite conjugacy class $\co$, we choose once and for all $w_{\co}\in \co^{\min}$ a representative of minimal length.

        By \cref{thm:minimal_thm_finite}, every element can be cyclically reduced to minimal length. 
        When \( w \xrightarrow{s} w' \), we have
        \[
        T_w \equiv 
        \begin{cases}
        T_{w'} & \text{if } \ell(w) = \ell(w'), \\
        b_s T_{w'} + a_s T_{sw} & \text{if } \ell(w) > \ell(w')
        \end{cases}
        \mod [\mathcal{H}, \mathcal{H}].
        \]
        
        Indeed for the second case, write \( w' = sws \), we compute
        \[
        T_s T_w = T_s(b_s T_{sws} + a_s T_{sw}) = b_s T_{ws} + a_s T_w = T_w T_s.
        \]
        Hence, by \cref{thm:minimal_thm_finite}, we may inductively, express $T_w$ into linear combination of $T_{w_\co}$. The coefficients $f_{w,\co}\in R$ are called the \emph{class polynomials}. The class polynomials $f_{w,\co}$ satisfy the following properties. If $w$ is a minimal length element, then $f_{w,\co}=\delta_{\{w\in \co\}}$. For smaller length $w$, by \cref{thm:minimal_thm_finite}, $f_{w,\co}$ can be calculated inductively:
        \begin{equation} \label{eq:fmin-recursion}
            f_{w',\co}
              = \begin{cases}
                  f_{w,\co}, & \text{if } w \approx w', \\[6pt]
                  b_s\,f_{w,\co} \;+\; a_s\,f_{s w,\co}, & \text{if } w' \xrightarrow{s} w\text{ and } \ell(w')>\ell(w), s\in S.
                \end{cases}
        \end{equation}
        
        \begin{prop}\cite[Theorem 8.2.3]{geck2000characters}
            For each $w\in W$, we have
            \begin{equation}
                T_w\equiv \sum_{\co\in \cl(W)} f_{w,\co} T_{w_{\co}} \mod [\ch,\ch].
            \end{equation}
            Moreover, the class polynomials $f_{w,\co}$ are uniquely determined by this equation.
        \end{prop}
    \begin{rmk}
        By this proposition, the uniqueness of the class polynomial is equivalent to the linear independence of the basis elements (the image of) $\{T_{w_\co}\}_{\co}$ in the cocenter $\ch/[\ch,\ch]$.
    \end{rmk}

\section{Classification of finite partial conjugacy classes}
    Let $w\in W, J\subset S$. In this section, we classify when the $W_J$-conjugacy classes in $W$ are finite. 
    
    
     Denote by $\bar{w}\in {}^JW^J$ the unique minimal length element in $W_JwW_J$. 
     
     Denote $w\cdot J:=w J w \i$.
     
     Following \cite[Corollary 2.3,2.6]{he2007minimal}, setting $K_v=\max\{K \subset J\mid  v\cdot K=K\}$ and define the \textit{partial conjugation} of $W_J$ acting on $W$ by $x\cdot y=xyx^{-1}$. For an automorphism $\sigma: W\to W$, we may define the \textit{twisted conjugation} by $x\cdot_\sigma y=xy\sigma(x)\i$.

    Moreover, we have the following correspondence on the $\Ad(v)$-twisted conjugacy classes $\mathcal{C}$ in $W_{K_v}$ and the $W_J$-conjugacy classes in $W_J\cdot (vW_{K_v})$.

    \begin{prop} \label{thm:partialconjdecom}
    Let $J\subset S$.
    \begin{enumerate}
        \item We have $$W=\sqcup_{v \in  {}^JW} W_J \cdot (v W_{K_v}).$$

        \item For each $v\in {}^JW$, write $K=K_v$ and consider an automorphism
        $\sigma: W_K\to W_K$, $\sigma(x)=vxv\i$. For an $\sigma$-twisted conjugacy class $\mathcal{C}$ in $W_K$, we have $$W_J\cdot v\mathcal{C}=W_J \times^{W_K} v\mathcal{C}$$ 
    where the right-hand side is the quotient space given by $(w_J,vx)\sim (w_Jw_K\i,w_Kvxw_K\i)$, for $w_J\in W_J$, $x\in \mathcal{C},w_K\in W_K$. 
    \end{enumerate}
    \end{prop}

    \begin{proof}
        \begin{enumerate}
            \item This is \cite[Corollary 2.6]{he2007minimal}.
            \item Note that $\sigma=\Ad(v)$, and $xv=v\sigma\i(x)$ for any $x\in W_K$. 

            On the one hand, write $w_J=w_1w_K, w_1\in W_J^K$, $w_K\in W_K$. Then for $x\in \mathcal{C}$, $w_J\cdot vx=w_1w_Kvxw_K\i w_1\i=w_1\cdot v\sigma\i(w_K)xw_K\i\in w_1\cdot v\mathcal{C}$.

            On the other hand, assume $w_J\cdot vx=vx'$, for $w_J\in W_J,x,x'\in \mathcal{C}$. Then $w_J=vx'x\i v\i\in vW_Kv\i \subset W_K$. Moreover, $w_J\cdot vx=v\sigma\i(w_J)xw_J\i$. Hence, $x'=\sigma\i(w_J)xw_J\i\in \mathcal{C}$ and $w_J\in W_K$.
        \end{enumerate}
    \end{proof}

    \begin{cor} \label{cor:counting}
        In the setting above, $$\sharp W_J\cdot v\mathcal{C}=\sharp W_J /W_K\times \sharp\mathcal{C}$$ 
    \end{cor}

    We quote two results to be used in our proof.
    \begin{lem}\cite[Lemma 3.3]{marquis2024centre}\label{lem:MarquisComb}
        Let $w\in W$. If there exists a non-spherical subset $J\subset S$ so that $H:=J\cap \bar{w}\cdot J$ is spherical, then $U^+_J(w)$ is infinite.
    \end{lem}
    The proof is purely combinatorial.
    \begin{lem}\cite[Lemma 2.1]{capraceMarquis2013open}\label{lem:normalizer}
        Let $J\subset S$ be a subset so that all irreducible components are non-spherical. Then the normalizer $N_W(W_J)=W_J\times W_{J^\perp}=W_{J\sqcup J^\perp}$.
    \end{lem}
    The proof is root-theoretic.
    
    \begin{thm}\label{thm:classification} Suppose $(W,S)$ is an infinite Coxeter group and $J$ is irreducible. Then $\co_J(w)$ is finite only in the following cases. 
    \begin{enumerate}
        \item $J$ is spherical, i.e., $W_J$ is finite;
        \item $w\in W_{J^\perp}$;
        \item $J$ is affine, and $w=w_1 w_2$, $w_1\in W_{J^\perp}$, $w_2\in W_J$, and $w_2$ is a translation in $W_J$.
        
    \end{enumerate}
    Here, $J^\perp=\{s\in S\setminus J\ | \  sj=js\text{ for all } \ j\in J \}.$
    \end{thm}
    \begin{proof}
        In all three cases, it is clear $\co_J(w)$ is finite. Only case (iii) needs some illustration. In fact, if we consider the decomposition of an affine Weyl group $W_J=W_0\ltimes X$ into a finite Weyl group $W_0$ and the coroot lattice $X$, and write $\tilde{w}=wt^\mu, \tilde{w}\in W_J, w\in W_0, \mu\in X$ (written exponentially as $t^\mu$), we have $wt^\mu w^{-1} =t^{w\mu}$ and $t^{\mu'}t^\mu t^{-\mu'}=t^\mu$. Hence $(wt^\mu) w't^{\mu'} (wt^\mu)\i = ww'w\i t^{w(w'\i\mu-\mu+\mu')}$ and the infinitude of conjugacy class of $w't^{\mu'}$ only depends on $\{w'\i\mu-\mu\ |\ \mu\in X\}$. It is finite only for translation elements ($w'=e$).
        
        Conversely, we want to show that these are the only three cases that can occur. Suppose $\co_J(w)$ is finite. The first two cases are obvious, so we may assume $J\subset S$ to be irreducible and non-spherical, and $w\notin W_{J^\perp}$, and we want to reduce it to the case $(iii)$.

        First, we claim that all proper subsets of $J$ are spherical. Assume not, then we may find a minimal non-spherical subset $J'\subsetneq J$ so that all proper subsets of $J'$ are spherical. Moreover, $J'\sqcup J'^\perp\subsetneq J$ as $J$ is irreducible. We may conjugate $w$ inside $U_J^+(w)$ so that $\supp w=J$, otherwise, for any $s\in J\setminus \supp w$, either $\ell(sws)=\ell(w)+2$ or $s\in (\supp w)^\perp$ (see e.g. \cite[Lemma 2.37]{abramenko2008buildings}). But then this implies $w\notin N_W(W_{J'})=W_{J'\sqcup J'^\perp}$ by \cref{lem:normalizer}. Hence $\bar{w}\cdot J'\neq J'$ (otherwise $w$ will normalize $W_{J'}$). As a result, $H:=J'\cap \bar{w}\cdot J'\subsetneq J'$ is spherical. Then $U^+_{J'}(w)$ is infinite by \cref{lem:MarquisComb} and so does $U^+_J(w)$ and $\co_J(w)$.

        Next, we have $\bar{w}\cdot J= J$, for otherwise, $\bar{w}\cdot J\cap J$ would be spherical and by \cref{lem:MarquisComb}, $\co_J$ would be infinite. Let $v\in {}^JW$ so that $w\in W_J\cdot vW_{K_v}$ in \cref{thm:partialconjdecom}(i). By \cref{cor:counting},  $\sharp W_J /W_K$ being finite forces $K_v=J$, i.e., $v\cdot J=J$. Hence, $v\in N_W(W_J)=W_J \times W_{J^\perp}$. As $v\in {}^JW$, we have $v\in W_{J^\perp}$. In this case, $\Ad(v)=\id$, and up to $W_J$ conjugation, $w=vw_J$ for some $w_J\in W_J$. Therefore, the partial conjugacy class $\sharp \co_J(w)=\sharp \co_J(w_J)$ is finite if and only if $J$ is affine and $w_J$ is a translation \cite[lemma 4.6]{marquis2024centre}.
        \end{proof}

    \section{Infinitude of cyclic shift classes}
        In this section, we will show that the cyclic shift class is infinite precisely when the conjugacy class is infinite, generalizing results in \cite[Section 4]{marquis2024centre}.
        
        Let $X=X(W,S)$ be a Davis complex associated with the Coxeter group $(W,S)$ with metric $d$. Let $J\subset S$. Let $X_J$ be the subcomplex of $X$ with the $1$-skeleton being the Cayley graph of $W_J$ and only Coxeter cells $P_I$, $I\subset J$ spherical. Note that as the metric is piecewisely defined on each $P_J$, it makes sense to restrict the metric on each $P_J$, and they glue into a metric, denoted as $d_J$, on $X_J$, making $(X_J,d_J)$ into a Davis complex, which is again a CAT($0$)-space. Note that although $d_J$ is different from $\restr{d}{X_J}$ in general, if for $x,y\in X$, the geodesic $[x,y]\subset X_J$, then $d_{J,w}=d_w$ on $[x,y]$. For $w\in W$, define the three important functions 
        $$d_w:X\to \mathbb{R}, x\mapsto d(x,wx)$$
        $$d_{J,w}:X_J\to \mathbb{R}, x\mapsto d_J(x,wx)$$ $$d_{\Ch,w}:\Ch(X)\to \mathbb{N}, C\mapsto d_{\Ch}(C,wC)$$
        Note that when $J=S$, $d_{J,w}=d_w$. We will use the same notation $d_{\Ch}$ for the chamber distance of any Davis complex when the context is clear. We shall explain the basic properties of these three functions.
        The first two are displacement functions on Davis complexes. It is a result of CAT($0$)-space that displacement functions are convex, see for example \cite[II, 6.2(3)]{bridson2013metric}. In particular, both $\restr{d_w}{X_J}$ and $d_{J,w}$ are convex on $X_J$. The third function is a chamber-wise function related to the lengths of conjugates of $w$: for $D\in \Ch(X)$, $d_{\Ch,w}(D)=d_{\Ch}(D,wD)=\ell(\pi_w(D))$. In particular, the statement that two $s$-adjacent chambers $D$ and $D'$, with $d_{\Ch,w}(D)\leq d_{\Ch,w}(D')$ translates to $\pi_w(D')\xrightarrow{s} \pi_w(D)$. 

        We call a chamber $C\in \Ch(X)$ as \emph{$(w,J)$-local maximum} if every non-decreasing $J$-gallery is flat, i.e., for every $\Gamma=(D_0=C,D_1,\dots,D_k)$ in $\Ch(X)$, with $d_{\Ch,w}(D_i)\leq d_{\Ch,w}(D_{i+1})$ for all $i$, we have $d_{\Ch,w}(D_i)= d_{\Ch,w}(D_{i+1})$ for all $i$. 

    First, we quote a result on the relations of $d_{\Ch,w}$ and $d_w$.

    \begin{prop}\cite[Proposition 4.4]{marquis2024centre}\label{prop:chamber_to_metric}
        Let $X'=X'(W',S')$ be a Davis complex with metric $d_w'$. Let $w\in W'$. Let $C,D\in \Ch(X')$ be adjacent chambers separated by wall $m$. Let $x\in \int(C)$, $y:=r_m(x)\in \int(D)$ and $z:=[x,y]\cap m$. Then $d_{\Ch,w}(C)<d_{\Ch,w}(D)$ implies $d_w'$ is strictly increasing on $[z,y]$.
    \end{prop}

    This result is essentially due to the convexity of $d_w$ and the strong convexity of $X$.

    We also need to quote a lemma for the Davis complex.
    \begin{lem}[{\cite[Lemma 2.5]{marquis2024centre}}] \label{lem:extgeod}
        Assume that $(W,S)$ is infinite, irreducible, but all proper subsets of $S$ are spherical.
        \begin{enumerate}
        \item
        Let $x$ be the barycentre of a chamber $C$, let $m_1,\dots,m_k$ be the walls of $C$, and for each $i=1,\dots,k$, choose a point $x_i\in [x,r_{m_i}(x)]\setminus \{x\}$. Then the convex hull of $\{x_1,\dots,x_k\}$ contains an open neighbourhood of $x$.
        \item 
        Every geodesic segment in $X$ can be extended to a geodesic line.
        \end{enumerate}
    \end{lem}

    \begin{prop}\label{prop:localmaximum}
        Assume that all proper subsets of $J$ are spherical. Let $w\in W$. If there exists a $(w,J)$-local maximum, all $W_J$-conjugates of $w$ are of the same length.
    \end{prop}

    \begin{proof}
        Let $C$ be a $(w,J)$-local maximum and consider the set of $(w,J)$-local maxima $C_{\approx}:=\{D\in \Ch(X)\ |\ \text{there is a $J$-gallery } (D_0=C,D_1,\dots,D_k=D) \text{ with } d_{\Ch,w}(D_i)=d_{\Ch,w}(C) \text{ for all }i\}$. Denote by $x_D$ the barycenter of the chamber $D$. 
        
        First, we claim that the set $\{d_w(x_D)\ |\ D\in C_{\approx}\}$ is finite. As $d_{\Ch,w}$ is constant on $C_{\approx}$, the set $\{\pi_w(D)\ |\ D\in C_{\approx}\}$ is finite. Moreover, $\pi_w(D)=\pi_w(D')$ if and only if $D$ and $D'$ are in the same $Z_W(w)$-orbit, so $C_{\approx}$ is covered by finitely many $Z_W(w)$-orbits of chambers. But then $d_w$ is constant on $Z_W(w)$-orbits, the set is finite as claimed.

        Now we may modify $C$ inside $C_{\approx}$ such that $d_w(x_C)$ is maximal. Let $D_i$ be the $J$-adjacent chambers of $C$ separated by walls $m_i$, $i=1,\dots,k$. By assumption of $C$, we have $d_{\Ch,w}(D_i)\leq d_{\Ch,w}(C)$. Hence we have only two cases, either $d_{\Ch,w}(D_i)<d_{\Ch,w}(C)$ or $d_{\Ch,w}(D_i)= d_{\Ch,w}(C)$. We claim that only the second case will happen. In the first case, by \cref{prop:chamber_to_metric}, $d_w$ is strictly increasing on $[x_i,x_C]$ where $x_i:=[x_{D_i},x_C]\cap m_i$. In the second case, we have $d_w(x_{D_i})\leq d_w(x_{C})$ by maximality of $x_C$. As $d_w$ is convex, either it is strictly increasing on some $[x_i,x_C]$ for some $x_i\in [x_{D_i},x]\setminus\{x\}$ or it is constant on $[x_{D_i},x_C]$ (in which case we set $x_i=x_{D_i}$). Note that as $[x_{D_i},x_C]\subset X^{(1)}\subset X_J$, we may replace $d_w$ into $d_{J,w}$. If the first case does occur, then $d_{J,w}$ is strictly increasing on some $[x_j,x_C]$. By \cref{lem:extgeod}, we may extend $[x_j,x_C]$ to $[x_j,y]$ with $[x_j,x_C]\subsetneq [x_j,y]\subset X_J$, so that $d_{J,w}(x_i)\leq d_{J,w}(x_C)<d_{J,w}(y)$ for all $i$ and $y$ lies in the convex hull of $\{x_1,\dots,x_k\}$ inside $X_J$. But then this contracts to the convexity of $d_{J,w}$ on $X_J$.

        We have shown that for any chamber $D$ that is $J$-adjacent to $C$, we have $D\in C_{\approx}$ and $d_w(x_{D_i})=d_w(x_{C})$. We may inductively conclude this is the case for all chambers on a $J$-gallery from $C$. The proof is complete.

    \end{proof}

    \begin{thm}\label{thm:infU} Let $w\in W$ and $J\subset S$ irreducible. Then $\co_J(w)$ is infinite if and only if $U_J^+(w)$ is infinite.
    \end{thm}

    \begin{proof}
        The if direction is automatic. It remains to show the only if direction. Assume $\co_J(w)$ is infinite. Then $J$ must be non-spherical. 
        
        We may assume all proper subsets of $J$ are spherical, for otherwise, both $U_J^+(w)$ and $\co_J(w)$ are infinite by the argument in the proof of \cref{thm:classification}. 
        
        If $U_J^+(w)$ is finite, then the maximal length element inside $U_J^+(w)$ will produce a $(w,J)$-local maximum, which, by \cref{prop:localmaximum}, further implies $\co_J(w)$ is finite.
    \end{proof}

    \section{The centralizers $\mathcal{Z}_{\ch}(\ch_J)$}
        
        \subsection{The generic Hecke algebra}
            Let $(W,S)$ be a Coxeter group with the length function $\ell$. Let $R$ be a commutative ring. Let $\ch(W,S,(a_s,b_s)_{s\in S})$ be a \emph{generic Hecke algebra} as in the introduction.

            We may use another equivalent definition in practice, namely 
           \begin{equation*}
                T_sT_w = 
                \left\{
                \begin{aligned}
                    &a_s T_{w} + b_s T_{sw} & \text{if } \ell(sw) < \ell(w) \\
                    &T_{sw} & \text{if } \ell(sw) > \ell(w)
                \end{aligned}
                \right.
            \end{equation*}
            
            As centralizers of products are products of centralizers, we may assume $J$ is irreducible.
            \begin{ass}
                In the sequel, we assume $b_s$ are all invertible in $R$ and $J$ is irreducible.
            \end{ass} 
            As a consequence, $\{T_w\}_{w\in W}$ are all invertible. 
    
        Denote by $\ch=\ch(W,S,(a_s,b_s)_{s\in S})$ and $\ch_J=\ch(W_J,J,(a_s,b_s)_{s\in J})$ as a subalgebra of $\ch$.
        The centralizer (or \emph{partial center}) $\mathcal{Z}_{\ch}(\ch_{J})$ is a generalization of the center.

        \begin{prop}\label{prop: diamond}
            Write $x=\sum_{w\in W}x_wT_w$, $x_w\in R$. Then $x\in \mathcal{Z}_{\ch}(\ch_{J})$ if and only if
            \begin{enumerate}
                \item $x_w=x_{w'}$ whenever $w\approx_J w'$;
                \item $x_{w'}=b_sx_w-a_sx_{sw}$ whenever $w\xrightarrow{s} w'$ and $\ell(w)>\ell(w')$ for $s\in J$.
            \end{enumerate}
        \end{prop}

        \begin{proof}
            For $z\in \ch$ and $w\in W$, we use $(z)_w$ to denote the coefficient of the basis element $T_w$ in $z$, for example, $(x)_w=x_w$.

            We first show the only if part.
            \begin{enumerate}
                \item It suffices to show $x_w=x_{w'}$ for $w\xrightarrow{s} w'$ with $\ell(w')=\ell(w)$. We may assume $w'\neq w$.
                
                By the exchange property of Coxeter groups, we cannot have both $\ell(sw)<\ell(w)$ and $\ell(ws)<\ell(w)$. Indeed, $\ell(sw)<\ell(w)$ implies there must exist a reduced expression $w$ starting with $s$, say $w=s_0s_1s_2\cdots s_k$ with $s=s_0$. But then $\ell(ws)<\ell(w)$ implies $ws=s_0s_1\cdots \hat{s_j}\cdots s_k$ with $\hat{}$ denoting omission. If $j\neq 0$ then $\ell(sws)<\ell(w)$. If $j=0$, then $ws=sw$ and that further implies $w'=sws=ssw=w$.
                
                Hence either $\ell(sw)>\ell(w)=\ell(w')$ or $\ell(ws)>\ell(w)=\ell(w')$.

                In the first case, $(T_sx)_{sw}=x_w+a_sx_{sw}$ and $(xT_s)_{sw}=x_{sws}+a_sx_{sw}$. Hence $x_{w'}=x_w$. The second case is similar.

                \item In this case, we have $\ell(sws)<\ell(ws)=\ell(sw)<\ell(w)$ and $sw\approx_J ws$. Then $(T_sx)_{sw}=b_sx_w$ and $(xT_s)_{sw}=a_sx_{sw}+x_{sws}$. Hence $x_{w'}=b_sx_w-a_sx_{sw}$. 
            \end{enumerate}

            Now we show the if part, i.e., for $s\in J$, we want to show $T_sx=xT_s$ when $x_w$, $w\in W$ satisfy both conditions. For each $w\in W$, $s\in J$, there are two cases. If $|\ell(sws)-\ell(w)|=2$ or ($\ell(sws)=\ell(w)$ and $|\ell(ws)-\ell(sw)=2|$), we are in the situation of $(ii)$. By the proof in $(i)$ above, the case $\ell(w)=\ell(sws)>\ell(ws)=\ell(sw)$ is trivial. The only case left is $\ell(w)=\ell(sws)<\ell(ws)=\ell(sw)$. But then this implies $sws\approx_J w$, and $sw\approx_J ws$. We are in the situation of $(i)$. Hence $T_sx=xT_s$ for any $s\in J$.
        \end{proof}
    \begin{rmk}
        The proof does not assume $\{b_s\}_{s\in S}$ are invertible. 
    \end{rmk}
    \begin{rmk}
        Note that in (ii), we also have $x_{sw}=x_{ws}$ by (i).
    \end{rmk}
    \begin{rmk}
        It seems likely $(i)$ holds for a stronger condition for $w\sim_J w'$. Moreover, the reduction of $T_w$ inside the cocenter also follows relations similar to the class polynomials. This observed relationship can be formalized through the duality between the Hecke algebra $\ch$ and its completion $\widehat{\ch}$ (direct product of $W$-many basis elements). The partial cocenter with respect to the parabolic subalgebra $\ch_J$ inside $\ch$ is naturally dual to the centralizer $Z_{\widehat{\ch}}(\ch_J)$ of $\ch_J$ within the completion $\widehat{\ch}$. This generalizes the duality via the symmetrizing form discussed in \cite{geck1997centers}.
    \end{rmk}

    We aim to define a maximal version of the class polynomial to record the coefficients of all $T_w$ in the reduction given in \cref{prop: diamond}. For that, we need a foundation for reducing elements to their maximal lengths. 
    
    \subsection{Reduction theorem and variations}
        We want to state the reduction theorem and its variations, generalizing the result in the finite case \cite{geck1993irreducible} and the affine case \cite{he2014minimal}.

        Let $(W,S)$ be a Coxeter system. Let $\Aut(W,S)=\{\delta\in \Aut(W)\ |\ \delta(S)=S\}$. Let $\tilde{W}=W\rtimes \Aut(W,S)$. For $J\subset S$ and $\delta\in \Aut(W_J,J)$, we can define the (partial) $\delta$-twisted conjugation $x\cdot y=xy\delta(x)\i$ for $x\in W_J,y\in W$. We could also define $\delta$-twisted conjugacy class $\co_\delta$, $\delta$-cyclic shift $w\rightarrow_\delta w'$, $\delta$-strong conjugation $w\sim_\delta w'$ and their partial versions $w\rightarrow_{\delta,J} w'$, $w\sim_{\delta,J} w'$, see for example, \cite[\S 3.1]{he2007minimal}. We use the word ``untwisted'' for cases when $\delta=\id$.
        
        \begin{thm}[\cite{marquis2018cyclically, marquis2020structure}]\label{thm:mingeneral}
             Let $(W,S)$ be a Coxeter system and let $\delta\in \Aut(W,S)$. Let $\co_\delta$ be a $\delta$-twisted conjugacy class of $W$, and let $\co^{\min}_{\delta}$ be the set of minimal length elements of $\co_\delta$. Then the following assertions hold:
            \begin{enumerate}
                 \item For each $w \in \co_{\delta}$, there exists $w' \in \co^{\min}$ such that
            $w \rightarrow_{\d} w'$.
            
                \item Let $w, w' \in \co^{\min}_\delta$, then $w \sim_{\d} w'$.
            \end{enumerate}
        \end{thm}

        \begin{rmk}
            The statement in \cite{marquis2018cyclically} is proved for untwisted conjugacy classes in $W$. In $(ii)$, they use tight conjugation, which is a special type of strong conjugation. By \cite[Remark 6.9]{marquis2020structure}, the untwisted statements can be upgraded to work for $\tilde{W}$. By \cite[Remark 3.6]{marquis2020structure}, the statement can be reduced to a twisted statement in $W$ as above.
        \end{rmk}

        The above result can be made into the partial conjugation version via the following trick.
        
        \begin{cor}Let $(W,S)$ be a Coxeter system and $J\subset S$. Let $\delta\in \Aut(W_J,J)$. Let $\co_{\delta,J}$ be a $\delta$-twisted $W_J$-conjugacy class of $W$, and let $\co^{\min}_{\delta,J}$ be the set of minimal length elements of $\co_{\delta,J}$. Then the following assertions hold:
            \begin{enumerate}
                 \item For each $w \in \co_{\delta,J}$, there exists $w' \in \co_{\delta,J}^{\min}$ such that
            $w \rightarrow_{\d,J} w'$.
            
                \item Let $w, w' \in \co^{\min}_{\delta,J}$, then $w \sim_{\d,J} w'$.
            \end{enumerate}
        \end{cor}
        \begin{proof}
            Define a new Coxeter group $(\hat{W}, S\sqcup\{s_{J,\infty}\})$ with $s_{J,\infty} j=js_{J,\infty}$ for all $j\in J$ and $s_{J,\infty} s$ having infinite order for $s\in S\setminus J$. Consider the embedding $\hat{}: W\to \hat{W}$, $w\mapsto \hat{w}:=ws_{J,\infty}$. Then $\hat{w}\rightarrow_\d \hat{w'}$ in $\hat{W}$ is equivalent to $w \rightarrow_{\d,J} w'$, and $\hat{w}\sim_\d \hat{w'}$ in $\hat{W}$ is equivalent to $w \sim_{\d,J} w'$. Therefore, applying \cref{thm:mingeneral} to $\hat{W}$ will give the desired result.
        \end{proof}

    We can also have the maximal length element version of the above corollary. Before doing this, we need a variant of the relation $\sim$. (Notice the difference between $\sim$ and $\backsim$.)
    \label{defmaxequ}
    
    Let \( w, w' \in W \) be such that \( \ell(w) = \ell(w') \). We shall write \( w \overset{x}{\backsim} w' \) if \( w' = xwx\i \) for some \( x \in W_J \) and if \( \ell(xw) = \ell(w) - \ell(x) \) or \( \ell(wx\i) = \ell(w) - \ell(x) \). 
    Again, we also write \( w \backsim_J w' \) if there exist sequences of elements \( x_i\in W_J, w_i \in W \) (\( 1 \leq i \leq n \)) such that \[ w \overset{x_1}{\backsim} w_1 \overset{x_2}{\backsim} \dots \overset{x_n}{\backsim} w_n = w'. \]

    Now we want to prove the maximal length theorem. We need two lemmas. We will use the following setting.
    
    Assume $J\subset S$ to be spherical, and $w'\in W$. Then the double closet $W_Jw'W_J$ contains a unique maximal element, denoted as $w$. Then any element $u\in W_Jw'W_J$ is of the form $u=w_1ww_2$, $w_1,w_2\in W_J$ with $\ell(u)=\ell(w)-\ell(w_1)-\ell(w_2)$. For a finite conjugacy class $\co_J$, denote the set of maximal length elements as $\co_J^{\max}$.
    \begin{lem}\label{lem:maxdoublecoset}
        Let $u=w_1ww_2\in W_JwW_J$, and $v=ww_2w_1\in wW_J$. Denote the conjugacy class of $u,v$ by $\co_J$. Then $u\in \co_J^{\max}$ if and only if $v\in \co_J^{\max}$ and $\ell(w_2w_1)=\ell(w_2)+\ell(w_1)$.

        In fact, if we write $w_1=s_1s_2\cdots s_k$, then $u\in \co_J^{\max}$ if and only if \\ $s_{t+1}\cdots s_{k-1}s_kww_2s_1s_2\cdots s_t\in \co_J^{\max}$ and is reduced for all $t$.
    \end{lem}
    \begin{proof}
        As $\ell(v)=\ell(w)-\ell(w_2w_1)$, and $\ell(w_2w_1)\leq \ell(w_2)+\ell(w_1)$. Hence $\ell(u)\leq \ell(v)$ with equality if and only if $\ell(w_2w_1)=\ell(w_2)+\ell(w_1)$. The same proof works for the in fact part.
    \end{proof}

    \begin{lem} Let $K:=\{s\in J\ |\ wsw\i\in W_J \}$. Then \label{lem:maxWK}
    \begin{enumerate}
        \item $K=\{s\in J\ |\ wsw\i\in J \}$;
        \item $W_K=\{x\in W_J\ |\ wxw\i\in W_J \}$
    \end{enumerate}
    Hence, $\d:=\Ad(w):(W_K,K)\to (W_{\d(K)},\d(K))$ is an isomorphism of Coxeter systems.
    \end{lem}
    \begin{proof}
        \begin{enumerate}
            \item Let $v=wsw\i\in W_J$. Then $\ell(vw)=\ell(ws)=\ell(w)-1$. Then $\ell(v)=1$.
            \item Follows from $(i)$ as simple reflections in $W_K$ is necessarily contained in $K$.
        \end{enumerate}
    \end{proof}

    \begin{cor}\label{cor:maximalthm}
    Let $(W,S)$ be a Coxeter system and $J\subset S$. Let $\co_{J}$ be a $W_J$-conjugacy class of $W$. Then the following assertions hold:
            \begin{enumerate}
                 \item If $\co_{J}$ is finite, then for each $u \in \co_{J}$, there exist a finite chain $u_k\xrightarrow{s_k}u_{k-1}\xrightarrow{\dots}u_1\xrightarrow{s_1}u$, where $u_k \in \co^{\max}_J$ and all $s_j\in J$.
                 \item If $\co_{J}$ is infinite, then for each $u \in \co_{J}$, there exist an infinite chain of distinct elements $\xrightarrow{\dots}u_2\xrightarrow{s_2}u_1\xrightarrow{s_1}u$, where all $s_j\in J$.
            
                \item For any $u, u' \in \co^{\max}_{J}$, we have $u \backsim_{J} u'$.
            \end{enumerate}
        \end{cor}

    \begin{proof}
        For infinite $\co_J$, \cref{thm:infU} implies $(ii)$.
        
        According to the classification \cref{thm:classification}, either $J\subset S$ is spherical or affine.
        Let $J$ be spherical and $u\in W$. Write $w$ for the unique longest element in $W_JuW_J$. Then $u=u_1wu_2$, for $u_1,u_2\in W_J$. We have $wu_2u_1\rightarrow_J u$ as $\ell(wu_2u_1)=\ell(w)-\ell(u_2u_1)\geq \ell(w)-\ell(u_2)-\ell(u_1)=\ell(u)$. We may thus assume $u=wy\in wW_J$. By \cref{lem:maxdoublecoset}, we may assume $u'\in\co^{\max}_J$ are of the form $u'=wy'\in wW_J\cap \co^{\max}_J$. Now for $x\in W_J$, $$wy=xwy'x\i \iff y=(w\i xw)y'w(w\i xw)\i w\i=zy'(\Ad(w)\cdot z)\i,$$
        where $z:=w\i xw$. By \cref{lem:maxWK}, $z\in W_K$ and if we define $W_K$-action on $W_J$ by $z\cdot_{\Ad(w)}y =zy(\Ad(w)\cdot z)\i$, called $\Ad(w)$-twisted conjugation, then we have $wy$ is $W_J$-conjugate to $wy'$ if and only if $y$ is $\Ad(w)$-twisted $W_K$-conjugate to $y'$. Thus, $wy'\rightarrow_J wy\iff y\rightarrow_{K,\Ad(w)} y'$. Similarly $wy'\backsim_J wy\iff y\sim_{K,\Ad(w)} y'$ by \cref{lem:maxdoublecoset}.

        Let $J\subset S$ be affine. By the proof of \cref{thm:classification}, every conjugate of a translation element is of the same length. The claim holds trivially.
    \end{proof}
    \subsection{A variation of class polynomial}
        Now we define a variation of the class polynomial $f_{w,\co_J}$ in \cite[\S 8.2]{geck2000characters}. Let $w\in W$, and $\co_J$ be a finite $W_J$-conjugacy class. According to \cref{thm:classification}, there are three cases. For $w\in W_J$, $w_1\in W_{J^\perp}$, define $f^{\max}_{ww_1,\co_J}=f^{\max}_{w,\co_J}$. If $J$ is affine, then $f^{\max}_{w,\co_J}=1$ if $w\in \co_J$ and $f^{\max}_{w,\co_J}=0$ otherwise. If $J$ is spherical, we define $f^{\max}_{w,\co_J}$ inductively as follows.
        If $w$ is a maximal length element, then $f^{\max}_{w,\co_J}=1$ for $w\in \co_J$ and $f^{\max}_{w,\co_J}=0$ otherwise.  For a smaller length $w$ we inductively define 
        \begin{equation} \label{eq:fmax-recursion}
            f^{\max}_{w',\co_J}
              = \begin{cases}
                  f^{\max}_{w,\co_J}, & \text{if } w \sim_J w', \\[6pt]
                  b_s\,f^{\max}_{w,\co_J} \;+\; a_s\,f^{\max}_{s w,\co_J}, & \text{if } w' \xrightarrow{s} w\text{ and } \ell(w')>\ell(w), s\in J.
                \end{cases}
        \end{equation}
        We set $f^{\max}_{w,\co_J}=0$ for all other cases.\\

        Note that the inductive formula for $f^{\max}_{w,\co_J}$ is the same as that for the class polynomial $f_{w,\co_J}$ in \cref{eq:fmin-recursion}, except we change $\sim$ to $\backsim$. The main difference of $f^{\max}_{w,\co_J}$ and $f_{w,\co_J}$ is the initial value. Here, we used the assumption that $b_s$ is invertible.

        Note also that $\approx_J$ implies $\backsim_J$, so in particular, $f^{\max}_{s w,\co_J}=f^{\max}_{ws,\co_J}$ in the case $w' \xrightarrow{s} w$, $\ell(w')>\ell(w)$.

        With a slight abuse of terminology, we also call $f^{\max}_{w,\co_J}$ the (max version of) \emph{class polynomial}. This polynomial shares similar properties to those in \cite[\S 8.2]{geck2000characters}.

        Apriori, the class polynomial $f^{\max}_{w,\co_J}$ depends on each $s$ used to conjugate $w$ to its maximal length. But we shall see that this would not happen.

        \begin{lem}\label{lem:sameimg}
            If $w\backsim_J w'$, then $T_w\equiv T_{w'} \mod [\ch_J,\ch]$.
        \end{lem}
        \begin{proof}
            Write $w'=xwx\i$ with $\ell(w)=\ell(w')$ and $x\in W_J$. We have either $\ell(xw)=\ell(w)-\ell(x)$ or $\ell(wx\i)=\ell(w)-\ell(x)$. In the first case, $T_w=T_{x\i}T_{xw}$ and $T_{w'}=T_{w'x}T_{x\i}$. Hence $T_w-T_{w'}=T_{x\i}T_{xw}-T_{xw}T_{x\i}\in [\ch_J,\ch]$. The second case is similar.
        \end{proof}

        Let $\cl^{\fin}_J(W)$ be the set of all finite $W_J$ conjugacy classes in $W$. For each finite conjugacy class $\co_J$, we choose once and for all $w_{\co_J}\in \co^{\max}_J$ a representative of maximal length.
        \begin{lem}\label{lem:Tw_and_class_poly}
            For $w\in W$, and $J$ is spherical or affine, we have
            \begin{equation}
                T_w\equiv \sum_{\co_J\in \cl^{\fin}_J(W)} f^{\max}_{w,\co_J} T_{w_{\co_J}} \mod [\ch_J,\ch],\label{eqn:linearcomb}
            \end{equation}
            where the sum is finite.
        \end{lem}
        
        \begin{proof}
            
            First for $w, w' \in \co^{\max}_{J}$, then $T_w\equiv T_{w'} \mod [\ch_J,\ch]$ by \cref{cor:maximalthm} and \cref{lem:sameimg}.

            According to the inductive definition of $f^{\max}_{w,\co_J}$ is not zero only when it is part of some cyclic shift process. When $J$ is spherical, $f^{\max}_{w,\co_J}$ is only related to $f^{\max}_{u,\co_J}$ for $u\in W_JwW_J$ which is a finite set. When $J$ is affine, the sum is just one term. Hence, the sum above is always finite.
            
            As the affine case is trivial, it remains to prove for the spherical case. For $w$ is not of maximal length, by \cref{cor:maximalthm}, there exist $w'\approx w$ and $s\in J$ so that $\ell(sw's)> \ell(w')=\ell(w)$. By \cref{lem:sameimg}, we have $T_w\equiv T_{w'}\mod [\ch_J,\ch]$. We may treat $w$ as $w'$ and prove for the case $\ell(sws)>\ell(sw)=\ell(ws)>\ell(w)$. We claim that $$T_{sws}\equiv b_sT_{w}+ a_sT_{sw} \mod [\ch_J,\ch].$$ Indeed, $ T_sT_{sw}-T_{sw}T_s=a_sT_{sw}+b_sT_w-T_{sws} \mod [\ch_J,\ch]$. But then this is exactly the definition of $f^{\max}_{w,\co_J}$.
        \end{proof}

        \begin{prop}\cite{partialcocenterADLV}\label{prop:cocenterbasis} 
        For spherical $J$, $\mathcal{H}/[\mathcal{H}_J,\mathcal{H}]$ has a basis given by $\{T_{w_{\co_J}}\}_{\co_J\in\cl_J(W)}$.
        \end{prop}

        \begin{cor}
            The $f^{\max}_{w,\co_J}$ is uniquely determined by the \cref{eqn:linearcomb}.
        \end{cor}

        \begin{proof}
            For $J$ affine, every element in a finite conjugacy class is of the same length, so no reduction will occur. For $J$ spherical, by \cref{lem:Tw_and_class_poly}, the uniqueness of $f^{\max}_{w,\co_J}$ is equivalent to the linear independence of the basis given in \cref{prop:cocenterbasis}. 
        \end{proof}

        Denote $b_w=b_{s_1}b_{s_2}\cdots b_{s_k}$ for $w=s_1s_2\cdots s_k$ some/any reduced expression. This is well-defined by Matsumoto's theorem. See, for example, \cite[\S 1.2]{geck2000characters}.
        
        \begin{thm}\label{thm:mainthmcentralizer} A basis of the centralizer $\mathcal{Z}_{\mathcal{H}}(\mathcal{H}_{J})$ is given by
        $$\{z_{\mathcal{O}_J}=\sum_{w\in W} b_w^{-1} f^{max}_{w,\mathcal{O}_J} T_{w\i}\}_{\mathcal{O}_J\in \cl_J^{\fin}(W)}$$
        
        Here, each sum is a finite sum.
        \end{thm}

        \begin{proof}
            Each sum is finite as $\co_J$ is finite as $f^{max}_{w,\mathcal{O}_J}$ vanishes for almost all $w$.
            First, we check that $z_{\co_J}\in \mathcal{Z}_{\mathcal{H}}(\mathcal{H}_{J})$ as in \cref{prop: diamond}. By definition, $b_w$ stays the same for conjugacy classes of $w$, and also $f^{max}_{w,\mathcal{O}_J}$ is stable under $\approx_J$. When $sws\xrightarrow[]{s} w$ (thus $sw\i s\xrightarrow[]{s} w\i$) with $\ell(sws)>\ell(w)$, by definition of $f^{max}_{w,\mathcal{O}_J}$ at \cref{eq:fmax-recursion}, we have 
            $$b_w^{-1} f^{max}_{w,\mathcal{O}_J}T_{w\i}=b_sb_{sws}^{-1}f^{max}_{sws,\mathcal{O}_J}T_{sw\i s}-a_sb_{ws}^{-1}f^{max}_{ws,\mathcal{O}_J}T_{sw\i}.$$

            As each $w\in \co_J^{\max}$ only appears in one $z_{\co_J}$, it is obvious that $\{z_{\co_J}\}_{\co_J}$ are linearly independent. It remains to show  $\{z_{\co_J}\}_{\co_J}$ spans $\mathcal{Z}_{\mathcal{H}}(\mathcal{H}_{J})$. 

            Let $x=\sum_{w\in W}x_wT_w\in \mathcal{Z}_{\mathcal{H}}(\mathcal{H}_{J})$. Let $v$ be an element with $\ell(v)$ maximal and $x_v\neq 0$ in the expression of $x$. Then we claim that $v\in \co^{\max}_J$ in its conjugacy class $\co_J$. First, we see that $\co_J$ must be finite, for otherwise, by \cref{thm:infU}, there exist $v'\in W$ with $v'\rightarrow_{J} v$ and $\ell(v')>\ell(v)$, and \cref{prop: diamond} (ii) together with maximality of $v$ forces $x_v=0$. For the same reason, if $v\notin \co^{\max}_J$, then \cref{cor:maximalthm} implies there exist $v'\in \co^{\max}_J$ with $v'\rightarrow_{J} v$ and $\ell(v')>\ell(v)$ and this forces $x_v=0$. The claim follows. Repeat the argument for $x-b_{v\i}z_{\co_J}$, we see that $\{z_{\co_J}\}_{\co_J\in \cl^{\fin}_J(W)}$ spans $\mathcal{Z}_{\mathcal{H}}(\mathcal{H}_{J})$.
           
        \end{proof}

    \begin{rmk}
        We give another argument for spanning using Nakayama's lemma. We use a specialization argument in \cite{lusztig1983singularities}.

            We write $a_s$ as a shorthand for $\{a_s\}_{s\in S}$. Similarly for $b_s$. Consider a specialization $\ch\to \mathbb{Z}[W]$, $b_s\mapsto 1, a_s\mapsto 0$. Denote  the image of $z_{\co_J}$ under this map by $\overline{z_{\co_J}}$. Clearly $\overline{z_{\co_J}}$ spans $\mathcal{Z}_{\mathbb{Z}[W]}(\mathbb{Z}[W_{J}])$. Let $A=\mathbb{Z}[a_s,b_s,b_s\i]_{(b_s-1,a_s)}$. Then $A/(b_s-1,a_s)\cong\mathbb{Z}$. By Nakayama's lemma, $\{z_{\co_J}\}_{\co_J}$ spans  $\mathcal{Z}_{\mathcal{H}}(\mathcal{H}_{J})$ over $A$. As both $\{z_{\co_J}\}_{\co_J}, \mathcal{Z}_{\mathcal{H}}(\mathcal{H}_{J})\subset \ch$, the coefficients of $z\in \mathcal{Z}_{\mathcal{H}}(\mathcal{H}_{J})$ in  $\{z_{\co_J}\}_{\co_J}$ are necessarily in $\mathbb{Z}[a_s,b_s,b_s\i]$. The proof is complete.
    \end{rmk}
    \begin{rmk}
        One would expect similar results to work for twisted centralizers: for a Coxeter system isomorphism $\d:(W_J,J)\to(W_{J'},J')$ 
\[
\mathcal{Z}_{\ch,\d}(\ch_J) := \{ h \in \ch \mid h x = \d(x) h \text{ for all } x \in \ch_J \}.
\]
    The main difference in the proof would be the classification of finite twisted partial conjugacy classes and the basis for the twisted cocenter. This is beyond the scope of this paper.
    
    \end{rmk}

\bibliographystyle{plain}
\bibliography{main}
\end{document}